\documentclass[11pt]{article}
\usepackage[english]{babel}
\usepackage{amsmath,amsthm}
\usepackage[T1]{fontenc}
\usepackage{amsfonts}
\usepackage{mathrsfs}
\usepackage{color}
\usepackage{bbm}
\usepackage{enumerate}
\usepackage{amsrefs}
\usepackage{mathtools}
\usepackage{hyperref}
\usepackage{cite}

\newtheorem{theorem}{Theorem}[section]
\newtheorem{corollary}[theorem]{Corollary}
\newtheorem{lemma}[theorem]{Lemma}
\newtheorem{proposition}[theorem]{Proposition}
\newtheorem{assumption}{Assumption}
\theoremstyle{definition}

\theoremstyle{remark}

\newtheorem{example}[theorem]{Example}
\numberwithin{equation}{section}
\begin{document}
	
	\def\Pro{{\mathbb{P}}}
	\def\E{{\mathbb{E}}}
	\def\e{{\varepsilon}}
	\def\veps{{\varepsilon}}
	\def\ds{{\displaystyle}}
	\def\nat{{\mathbb{N}}}
	\def\Dom{{\textnormal{Dom}}}
	\def\dist{{\textnormal{dist}}}
	\def\R{{\mathbb{R}}}
	\def\O{{\mathcal{O}}}
	\def\T{{\mathcal{T}}}
	\def\Tr{{\textnormal{Tr}}}
	\def\sgn{{\textnormal{sign}}}
	\def\I{{\mathcal{I}}}
	\def\A{{\mathcal{A}}}
	\def\H{{\mathcal{H}}}
	\def\S{{\mathcal{S}}}
	
	\title{Global solutions to the stochastic heat equation with superlinear accretive reaction term and polynomially growing multiplicative noise.}%
	\author{M. Salins\\ Boston University \\ msalins@bu.edu }
	\maketitle
	
	\begin{abstract}
		We prove that mild solutions to the stochastic heat equation with superlinear accretive forcing and polynomially growing multiplicative noise cannot explode under two sets of assumptions. The first set of assumptions allows both the deterministic forcing and multiplicative noise terms to grow polynomially, as long as the multiplicative noise is sufficiently larger. The second set of assumptions imposes an Osgood condition on the deterministic forcing and allows the multiplicative noise to grow polynomially. In both cases, the multiplicative noise cannot grow faster than $u^{\frac{3}{2}}$, as this would cause explosion. 
		
	\end{abstract}
\section{Introduction}
We prove that the nonlinear stochastic heat equation
\begin{equation} \label{eq:SPDE}
	\frac{\partial u}{\partial t}(t,x) = \frac{\partial^2 u}{\partial x^2}(t,x) + b(u(t,x)) + \sigma(u(t,x))\dot{W}(t,x)
\end{equation}
has a global solution under certain assumptions when $b$ and $\sigma$ both grow superlinearly.
The spatial domain is $D:=[-\pi,\pi]\subset \mathbb{R}$ with periodic boundary conditions imposed. $\dot{W}(t,x)$ is a space-time white noise. We are primarily interested in two examples for the forcing term $b$: the polynomial case where $b(u) \approx u^\beta$ for $\beta \in (1,2]$, which with additive noise would cause explosion, and the Osgood case where $b(u) \approx u\log(u)$, which with additive noise would never explode \cite{bg-2009,fn-2021}. In both cases, we will show that if the multiplicative noise term $\sigma(u)$ satisfies appropriate lower and upper bounds, then the solutions to \eqref{eq:SPDE} can never explode. In the polynomial case, this proves that superlinear multiplicative stochastic noise can prevent explosion in SPDEs.

Mueller first investigated finite time explosion for equation \eqref{eq:SPDE} in the case where $b\equiv 0$ \cite{mueller-1991,mueller-1998,mueller-2000,ms-1993}. Mueller proved that when $b\equiv 0$, and $\sigma(u) \leq C(1 + |u|^\gamma)$, for $\gamma <\frac{3}{2}$, the solutions cannot explode. Furthermore, when $\sigma(u) =c |u|^\gamma$ for $\gamma>\frac{3}{2}$, solutions can explode with positive probability. Recently, it was shown that in the critical case when $\gamma=\frac{3}{2}$, solutions to the stochastic heat equation cannot explode \cite{s-2024-aop}. Similar results have been investigated for other spatial domains, noises, and SPDEs \cite{krylov,bezdek-2018,mueller-1997,bd-2002}

Foondun and Nualart, extending a result by Bonder and Groisman \cite{fn-2021,bg-2009}, proved that when $\sigma$ is bounded away from $0$ and $\infty$ in the sense that $0<c \leq \sigma(u) \leq C < +\infty$, and $b$ is positive and increasing, the so-called Osgood condition on $b$ fully characterizes explosion. Solutions explode with probability one if 
\begin{equation} \label{eq:Osgood-explode}
	\int_1^\infty \frac{1}{b(u)}du<\infty,
\end{equation}
and solutions exist for all time with probability one if 
\begin{equation} \label{eq:Osgood-no-explode}
	\int_1^\infty \frac{1}{b(u)}du = \infty. 
\end{equation}
This is the same condition that Osgood used to characterize explosion of the one-dimensional ordinary differential equation $\frac{dv}{dt} = b(v(t))$ \cite{osgood}. On unbounded domains, the explosive Osgood condition \eqref{eq:Osgood-explode}  guarantees instantaneous everywhere explosion \cite{fkn-2024}.

While explosion is fully characterized in the special cases where $b \equiv 0$ or $\sigma $ is bounded away from $0$ and $\infty$, the interaction of both superlinear $b$ and superlinear $\sigma$ is not fully understood. A first result in this direction is due to Dalang, Khoshnevisan, and Zhang \cite{dkz-2019}, who proved that if $b(u) \leq C( 1+ |u|\log|u|)$ and $\sigma(u) \in o(|u|(\log(u))^{\frac{1}{4}})$. Then solutions to \eqref{eq:SPDE} are global in time with probability one. This result has been extended slightly to other cases, including a more general Osgood condition on $b$, more general domains and colored noises, and to other settings like the stochastic  wave equation and stochastic heat equation on an unbounded spatial domain \cite{ms-wave-2021,ch-2023,salins-2022,sz-2022,lz-2022,av-2023}. All of the previously mentioned results about superlinear $\sigma$, however, consider $\sigma$ that grow like $|u|(\log|u|)^\gamma$, which is much slower than the allowable $|u|^{\frac{3}{2}}$ growth rate identified by Mueller in the $b\equiv 0$ case.

%
%

The results of this paper prove that Mueller's $\sigma(u)\approx |u|^{\frac{3}{2}}$ growth rate is the maximal growth rate on $\sigma$ to guarantee global solutions, even when $b$ is accretive and superlinear. We will prove that if $b$ satisfies the non-explosive Osgood condition \eqref{eq:Osgood-no-explode} and $\sigma(u) \leq C(1 + |u|^\frac{3}{2})$, then the mild solutions to \eqref{eq:SPDE} are global in time with probability one. This proves that the $\sigma \in o(u(\log(u))^{\frac{1}{4}})$ restriction from \cite{dkz-2019} is not a necessary condition for non-explosion.  
Additionally, we can prove that superlinear $\sigma$ can actually prevent explosion, even in some cases where $b$ satisfies the explosive Osgood condition \eqref{eq:Osgood-explode}, including the polynomial growth example $b(u)= u^\beta$ for $\beta \in (1,2]$. In the case where $b(u) = u^\beta$, we require $\sigma$ to grow faster than $ |u|^{\frac{\beta + 1}{2}}$, and to grow slower than $|u|^{\frac{3}{2}}$ .  We consider the following two sets of assumptions.

\begin{assumption} \label{assum}
	Assume that $b: \mathbb{R} \to \mathbb{R}$ and $\sigma:\mathbb{R} \to \mathbb{R}$ are locally Lipschitz continuous. Additionally, we assume either
	\begin{enumerate}
		\item[(a)] There exist constants $\theta>0$ and $C>0$, and a positive, increasing, convex $h: [0,+\infty) \to [0,+\infty)$ such that
		\begin{equation} \label{eq:b-upper-case-a}
			|b(u)| \leq \theta(1 + |u|) + h(|u|), \text{ for all } u \in \mathbb{R},
		\end{equation}
		\begin{equation} \label{eq:sigma-upper-lower-case-a}
			\left(\frac{1}{2\pi} + |u|\right)h(|u|) \leq \frac{1}{4\pi} \sigma^2(u) \leq C( 1+ |u|^3), \text{ for all } u \in \mathbb{R},
		\end{equation}
		and
		\begin{equation}
			\label{eq:convexity-case-a}
			\frac{h(u)h''(u)}{(h'(u))^2} \leq 2, \text{ for all } u>0.
		\end{equation}
		\item[(b)] Or there exists positive constants $c>0$, $C>0$, and $\gamma \in (1/2,1)$, and a positive, increasing, convex function $h:[0,+\infty) \to [0,+\infty)$ such that 
		\begin{equation} \label{eq:b-upper-case-b}
			|b(u)| \leq h(|u|) \text{ for all } u \in \mathbb{R},
		\end{equation}
		\begin{equation} \label{eq:h-osgood-case-b}
			\int_1^\infty \frac{1}{h(u)}du=+\infty,
		\end{equation} 
		\begin{equation} \label{eq:h-growth-case-b}
			\limsup_{u \to +\infty} \frac{h(u^2)}{u^{2\gamma+1}}=0,
		\end{equation}
		and 
		\begin{equation} \label{eq:sigma-upper-lower-case-b}
			c|u|^{2\gamma} \leq \sigma^2(u) \leq C(1 + |u|^3) \text{ for all } u \in \mathbb{R}.
		\end{equation}
	\end{enumerate}
\end{assumption}

Notice that condition (b) imposes a non-explosive Osgood condition on $b$, while (a) requires the growth rate of $b$ to be dominated by $\sigma$ in an appropriate way. The convexity assumptions on $h$, including \eqref{eq:convexity-case-a}, may look technical, but they are satisfied in most natural examples. The assumption \eqref{eq:convexity-case-a} guarantees that the mapping $u \mapsto uh^{-1}(u)$ is convex, and enables the proof of Lemma \ref{lem:Jensen}, below. Imposing appropriate convexity assumptions on $h$, while requiring that $h(u)$ dominates $b(u)$ and $(2\pi +u)h(u)$ is dominated by $\sigma^2(u)$, enables us to use Jensen inequality arguments without imposing an unnatural convexity assumption on either $b$ or $\sigma$. The examples we have in mind are $h(u) = Au^\beta$ or $h(u) = A(1 + u) (\log(1+u))^\beta$ for some $A>0$ and $\beta >0$. Both of these examples satisfy the convexity assumption \eqref{eq:convexity-case-a}. Similarly, \eqref{eq:h-growth-case-b} for case (b) is a very natural restriction that is only included to remove pathological counterexamples. We are mostly interested in $h$ of the form $h(u) = (1+u)\log(1+u)$ or $h(u) = (e+u) \log(e+ u)\log\log(e+u)$, which satisfy \eqref{eq:h-growth-case-b}.

The main result of this paper is the following.
\begin{theorem}\label{thm:main}
	Assume Assumption \ref{assum}.
	Then there exists a unique, global  mild solution to \eqref{eq:SPDE}.
\end{theorem}

We present example applications of this theorem. The first example considers the case of polynomial forcing $b(u)\approx u^\beta$ and $\sigma(u)\approx u^\gamma$. The second example considers nonexplosive Osgood-type forcing like $b(u)\approx u\log(u)$ or $b(u) \approx u \log(u)\log\log(u)$, etc.

\begin{example}
	Assumption \ref{assum}(a) is satisfied by polynomials of the form
	\[b(u)=Au^\beta, \sigma(u)=u^\gamma \]
	where either $\beta + 1< 2\gamma \leq 3$ or $\beta + 1=2\gamma \leq 3$ and $A$ is sufficiently small.
	
	When combined with other results from the literature, Theorem \ref{thm:main} demonstrates that when $\beta \in (1,2]$ and $\gamma$ is sufficently small or sufficiently large, solutions can explode, but when $\gamma$ is in the range $[\frac{1+\beta}{2}, \frac{3}{2}]$, solutions can never explode. When $\sigma$ is bounded or grows slowly, the explosive superlinear force $b$ causes the $L^1$ norm of the solution to reach infinity in finite time, implying that the $L^\infty$ norm also reaches infinity. When $\sigma$ grows sufficiently quickly, the large stochastic fluctuations counteract the explosive force of $b$, and the $L^1$ norm remains finite. When the $L^1$ norm stays finite, the situation is similar to the case investigated by Mueller \cite{mueller-1991,mueller-1998,mueller-2000,ms-1993}: the $L^\infty$ norm will not explode if $\gamma\leq  \frac{3}{2}$, but the $L^\infty$ norm can explode if $\gamma>\frac{3}{2}$. Consider the following ranges for $\gamma$
	\begin{itemize}
		\item If $\beta \in (1,2]$ and $\gamma=0$, so that $\sigma$ is bounded away from $0$ and $\infty$, then the results of Bonder and Groisman and Foondun and Nualart prove that solutions will explode with probability one \cite{bg-2009,fn-2021}.
		\item If $\beta \in (1,2]$ and $\gamma \in \left(0,\frac{\beta+1 }{2}\right]$, then the problem of explosion remains open, to the best of my knowledge. I believe that solutions can explode with positive probability in this regime.
		\item If $\beta \in (1,2)$ and $\gamma \in (\frac{1+\beta}{2}, \frac{3}{2}]$, then Theorem \ref{thm:main}, above, proves that solutions cannot explode. 
		\item If $\beta \in (1,2]$ and $\gamma = \frac{1 + \beta}{2}$, then for sufficiently small $A$, Theorem \ref{thm:main} proves that solutions cannot explode.
		\item If $\beta \in (1,2]$ and  $\gamma> \frac{3}{2}$, then the results of Mueller \cite{mueller-2000} and a comparison principle \cite{mueller-1991-support,kotelenez-1992} proves that solutions can explode with positive probability. 
	\end{itemize}
	 
\end{example}

\begin{example}
	Assumption \ref{assum}(a) is satisfied by the example
	\begin{align}
		b(u) = u\log(u) \text{ and } A|u||\log(u)|^{\frac{1}{2}} \leq \sigma(u) \leq u^{\frac{3}{2}},
	\end{align}
	for sufficently large $A$, and therefore Theorem \ref{thm:main} proves that solutions cannot explode.
	Compare this to Dalang, Khoshnevisan, and Zhang, who proved that there is no explosion when $b(u)= u\log(u)$ and $\sigma(u) \leq C|u||\log(u)|^{\frac{1}{4}}$ using completely different arguments \cite{dkz-2019}. This example proves that the upper bound on the growth rate of $\sigma$ identified by \cite{dkz-2019} is not optimal. Strangely, these two results leave a gap. What happens if $|u||\log(u)|^\frac{1}{4}<\sigma(u)<|u||\log(u)|^{\frac{1}{2}}$? 
	
	Assumption \ref{assum}(b) is included to answer this question. Because $b(u) = u\log(u)$ satisfies a non-explosive Osgood condition \eqref{eq:Osgood-no-explode}, Theorem \ref{thm:main} says that solutions will not explode for any $\sigma$ satisfying $u^\gamma< \sigma(u) < u^{\frac{3}{2}}$ for any $\gamma>\frac{1}{2}$. When combined with \cite{dkz-2019}, this proves that when $b(u)=u\log(u)$ solutions will not explode as long as $\sigma(u)$ satisfies either
	\[\sigma(u) \leq C|u||\log(u)|^{\frac{1}{4}}\]
	or there exists $\gamma>\frac{1}{2}$ such that
	\[c|u|^\gamma < \sigma(u) \leq C(1  + |u|^{\frac{3}{2}}).\]
	In particular, the class includes all $\sigma$ of the form $\sigma(u)  = |u|^\gamma(\log|u|)^\theta$ for $0 \leq \gamma < \frac{3}{2}$ and $\theta \geq0$ and for $\gamma=\frac{3}{2}$ and $\theta =0$.
	
	If $b$ grows faster than $u\log u$, like $u\log(u)\log\log(u)$ or $u \log(u)\log\log(u)\log\log\log(u)$, and still satisfies the non-explosive Osgood condition \eqref{eq:Osgood-no-explode}, then combining the results of \cite{salins-2022} and Theorem \ref{assum} lead to the same conclusion. We present this claim as a corollary.
\end{example}

\begin{corollary} \label{cor:main}
	Assume that $b$ and $\sigma$ are locally Lipschitz continuous and that there exists a positive, increasing, convex $h: [0,+\infty) \to [0,+\infty)$ such that 
	$b$ satisifies \eqref{eq:b-upper-case-b}, $h$ satisfies the non-explosive Osgood condition \eqref{eq:h-osgood-case-b}, and $\sigma$ satisfies either that there exists $c>0$, $C_0>0$, and $\gamma_0>\frac{1}{2}$ such that
	\[c_0 |u|^{\gamma_0} \leq \sigma(u) \leq C_0( 1+ |u|^{\frac{3}{2}})\]
	or there exists $\gamma_1 \in (0,1/4)$ and $C_1>0$ such that
	\[|\sigma(u)| \leq C_1(1 + |u|^{1-\gamma}(h(|u|))^{\gamma}).\]
	Then there exists a unique global mild solution to \eqref{eq:SPDE}. 
\end{corollary}

Under either set of assumptions in Assumption \ref{assum}, the lower bound on $\sigma$ can be used to prove that the $L^1$ norm of local solutions stays finite. Then we use the methods from \cite{s-2024-aop} to prove that the upper bound on $\sigma$ and the finiteness of the $L^1$ norm together imply that the $L^\infty$ norm cannot explode. Mueller proved that when $\sigma(u)$ grows faster than $|u|^{\frac{3}{2}}$, the $L^\infty$ norm can explode with positive probability \cite{mueller-2000}. Therefore, the upper bound on $\sigma$ in Assumptions \ref{assum}(a)--(b) is optimal. 

For the purposes of this introduction, assume additionally that $b(0)\geq 0$, $u(0,x) \geq 0$,  and $\sigma(0)=0$ so that, by the comparison principle, mild solutions to \eqref{eq:SPDE} stay positive \cite{kotelenez-1992,mueller-1991-support}. In this case, the $L^1$ norm is local semimartingale satisfying
\begin{align*}
	I(t) =&\int_{-\pi}^\pi u(t,x)dx = \int_{-\pi}^\pi u(0,x)dx + \int_0^t \int_{-\pi}^\pi b(u(s,x))dxds \\
	&+ \int_0^t \int_{-\pi}^\pi \sigma(u(s,x))W(dxds),
\end{align*}
until the time of explosion. Formally apply Ito formula for $\log(1 + I(t))$ to see
\begin{equation}
	\log(1 + I(t)) = \log(1 + I(0)) + \int_0^t \int_{-\pi}^\pi \left(\frac{b(u(s,x))}{1 + I(s)} - \frac{\sigma^2(u(s,x))}{2(1 + I(s))^2} \right)dx ds + N(t)
\end{equation}
where $N(t)$ is a local martingale. We will prove in Section \ref{SS:subsec-Osgood-case} that Assumption \ref{assum}(a) implies that that integral $ \int \left(\frac{b(u(s,x))}{1 + I(s)} - \frac{\sigma^2(u(s,x))}{2(1 + I(s))^2} \right)dx$ is bounded by a non-random constant with probability one, implying 
\begin{equation}
	\log(1 + I(t)) \leq \log(1 + I(0)) + C_1 t + N(t),
\end{equation}
which cannot reach infinity in finite time because $N(t)$ is a nonnegative local martingale.

Under Assumption \ref{assum}(b), we will prove in Section \ref{SS:subsec-Osgood-case} that the Lebesgue integral is bounded by
\[\int_0^t \int_{-\pi}^\pi \left(\frac{b(u(s,x))}{1 + I(s)} - \frac{\sigma^2(u(s,x))}{(1 + I(s))^2} \right)dx ds
\leq \int_0^t \tilde{g}(\log(1 + I(s)))ds + Ct\]
for some nonnegative increasing function $\tilde{g}$ satisfying the non-explosive Osgood condition \eqref{eq:Osgood-no-explode}. One more application of Ito formula with the nonnegative, increasing, concave function $G(u) = \int_0^u \frac{1}{\tilde g(v)}dv$ implies
\begin{equation}
	G(\log(1 + I(t))) \leq G(\log(1 + I(0))) + Ct + N(t),
\end{equation}
which cannot reach infinity in finite time because $N(t)$ is nonnegative local martingale.

Notice that the upper bounds on $\sigma$ are not relevant for proving that the $L^1$ norms stay finite. Under either set of our assumptions on the lower bound of $\sigma$, the $L^1$ norm of the solution would stay finite, up to the time of explosion, even if $\sigma$ grows faster than $|u|^\frac{3}{2}$. The upper bound on $\sigma$ is needed to prove that the finiteness of the $L^1$ norms implies that the $L^\infty$ norm cannot explode. 

Once, we have proven that the $L^1$ norm remains finite, we can follow the proof of \cite{s-2024-aop} to prove that the $L^\infty$ norm cannot explode. The proof involves defining a sequence of stopping times that keep track of when the $L^\infty$ of the process triples or falls to a third of its previous value. Using moment estimates on the supremum norm and the quadratic variation of the $L^1$ norm, we can prove that the $L^\infty$ norm can only triple a finite number of times, and therefore solutions cannot explode.

In Section \ref{S:local-soln} we define local mild solutions, global mild solutions, and explosion, and we fix some notation. In Section \ref{S:positive}, we employ a trick proposed in \cite{mueller-1998} to add a strong positive force to the equation to keep solutions positive. In Section \ref{S:L1}, we prove that the $L^1$ of solutions stays finite. Different proofs are needed depending on whether we assume Assumption \ref{assum}(a) or (b). In Section \ref{S:Linfinity}, we prove the main non-explosion result, Theorem \ref{thm:main}. 

%
%
%

\section{Definition of local and global mild solution} \label{S:local-soln}
Let $D:=[-\pi,\pi]$ with periodic boundary be the spatial domain. Define the periodic heat kernel for $x \in [-\pi,\pi]$ by
\begin{equation}
	G(t,x) = \frac{1}{2\pi} + \frac{1}{\pi} \sum_{k=1}^\infty e^{-|k|^2 t} \cos(kx).
\end{equation}
For any $n \in \mathbb{N}$, define the cutoff versions of $b$ and $\sigma$ by

\begin{equation}
	b_n(u) = \begin{cases}
		b(-n) & \text{ if } u<-n,\\
		b(u) & \text{ if } u\in(-n,n)\\
		b(n) & \text{ if } u>n.
	\end{cases}
\end{equation}
\begin{equation}
	\sigma_n(u) = \begin{cases}
		\sigma(-n) & \text{ if } u<-n,\\
		\sigma(u) & \text{ if } u\in(-n,n)\\
		\sigma(n) & \text{ if } u>n.
	\end{cases}
\end{equation}
Because we assumed in Assumption \ref{assum} that $b$ and $\sigma$ are locally Lipschitz continuous, the cutoff functions $b_n$ and $\sigma_n$ are globally Lipschitz continuous. The mild solution to the cutoff SPDE
\begin{equation}
	\frac{\partial u_n}{\partial t}(t,x) = \frac{\partial^2}{\partial x^2} u_n(t,x) + b_n(u_n(t,x)) + \sigma_n(u_n(t,x))\dot{W}(t,x)
\end{equation}
is defined to be the solution to the integral equation
\begin{align}
	u_n(t,x) = &\int G(t,x-y)u(0,y)dy + \int_0^t \int G(t-s,x-y)b_n(u_n(s,y))dyds  \nonumber\\
	&+ \int_0^t \int G(t-s,x-y)\sigma_n(u_n(s,y))W(dyds).
\end{align}
Because $b_n$ and $\sigma_n$ are globally Lipschitz continuous, classical results prove that for each $n$, there exists a unique mild solution $u_n(t,x)$ \cite{dpz,dalang-1999}.
Furthermore, these solutions are all \textit{consistent} in the sense that for any $n<m$
\begin{equation}
	u_n(t,x) = u_m(t,x) \text{ for all } x \in D \text{ and } t \in [0, \tau^\infty_n]
\end{equation}
where
\begin{equation}
	\tau^\infty_n := \inf\left\{t>0: \sup_{x \in D} |u_n(t,x)| > n\right\}.
\end{equation}
The consistency is a consequnece of the uniqueness of each $u_n(t,x)$ and the fact that $b_n(u) = b_m(u)$ and $\sigma_n(u)=\sigma_m(u)$ for all $|u|\leq n$.

Define the \textit{explosion time} by
\begin{equation}
	\tau^\infty_\infty = \sup_n \tau^\infty_n.
\end{equation}
We can uniquely define a \textit{local mild solution} by
\begin{equation}
	u(t,x) = u_n(t,x) \text{ for all } x \in D, t \in [0,\tau^\infty_n].
\end{equation}
The local mild solution is well defined for all $t \in [0,\tau^\infty_\infty)$, but it generally cannot be extended beyond $\tau^\infty_\infty$.
If $\tau^\infty_\infty<+\infty$, then we say that $u(t,x)$ explodes in finite time. The local mild solution is called a global mild solution if
\begin{equation}
	\Pro \left(\tau^\infty_\infty =\infty \right)=1.
\end{equation}

Throughout the paper, the constant $C>0$ can refer to any positive constant and its value can change from line to line. An integral symbol without explicit bounds refers to integration over the spatial domain
\[\int v(x)dx := \int_{-\pi}^\pi v(x)dx.\]
We use $L^p$ to refer to the standard $L^p$ spaces over the spatial domain with the norms
\begin{equation}
	|v|_{L^p} := \left( \int |v(x)|^p dx\right)^{\frac{1}{p}} \text{ and } |v|_{L^\infty} := \sup_{x \in [-\pi,\pi]} |v(x)|.
\end{equation}

\section{Positve local mild solutions} \label{S:positive}
As mentioned in the introduction, the analysis of the $L^1$ norm is easiest when the solutions $u(t,x)\geq 0$ for all $t$ and $x$. In this case, the $L^1$ norm is a local semimartingale. We follow the method of Mueller \cite{mueller-1998} to add an additional force to the equation that guarantees that solutions stay nonnegative. Specifically, let $\alpha>3$ and let $v(t,x)$ be the local mild solution to 
\begin{equation} \label{eq:v}
	\frac{\partial v}{\partial t}(t,x) = \frac{\partial^2 v}{\partial x^2}(t,x) + b(v(t,x)) + (v(t,x))^{-\alpha} + \sigma(v(t,x))\dot{W}(t,x)
\end{equation}
with initial data $v(0,x) = \max\{u(0,x),1\}$. 

We also define $v_-(t,x)$ to be the local mild solution to
\begin{equation} \label{eq:v-}
	\frac{\partial v_-}{\partial t}(t,x) = \frac{\partial^2 v_-}{\partial x^2}(t,x) - b(-v_-(t,x)) + (v(t,x))^{-\alpha} - \sigma(-v_-(t,x))\dot{W}(t,x)
\end{equation}
with initial data $v_-(0,x) = \max\{-u(0,x),1\}$.

A result from \cite{mueller-1998} proves that both $v_-(t,x)$ and $v(t,x)$ stay nonnegative until their explosion time.
Furthermore, the comparison principle of \cite{mueller-1991-support,kotelenez-1992}
 implies that $v_-(t,x) \leq u(t,x) \leq v(t,x)$ up until the time that $v$ or $v_-$ explodes. Therefore, it is sufficient to prove that these constructed positive solutions $v(t,x)$ and $v_-(t,x)$ never explode.
 
 The analysis for $v_-$ and $v$ are identical, so for the rest of the paper, we only analyze $v(t,x)$. Define the following stopping times  $v(t,x)$. Define for $n>0$,
 \begin{equation}
 	\tau^\infty_n := \inf\{t>0: |v(t)|_{L^\infty}>n\}
 \end{equation}
 and
 \begin{equation}
 	\tau^\infty_\infty: = \sup_n \tau^\infty_n.
 \end{equation}
 Define for $\e \in (0,1)$,
 \begin{equation}
 	\tau^{\inf}_\e: = \inf\left\{t>0: \inf_x v(t,x) <\e \right\}.
 \end{equation}
 and
 \begin{equation}
 	\tau^{\inf}_0 := \sup_{\e \in (0,1)} \tau^{\inf}_\e.
 \end{equation}
 
 Using standard localization procedures, there exists a unique local mild solution for $v(t,x)$ that is defined until the time $\tau^\infty_\infty \wedge \tau^{\inf}_0$ \cite{s-2024-aop}. Furtheremore, the local mild solutions satisfy the comparison principle of \cite{mueller-1991-support,kotelenez-1992}
 \begin{proposition}\label{prop:comparison}
 	The local mild solutions $u(t,x)$ and $v(t,x)$ satisfy the comparison principle
 	\begin{equation}
 		\Pro \left(u(t,x) \leq v(t,x) \text{ for all } x \in D, t \in [0,\tau^\infty_\infty \wedge \tau^{\inf}_0]\right)=1.
 	\end{equation}
 \end{proposition}
 
 The results of Mueller prove that because we added $v^{-\alpha}$ for $\alpha>3$ to the equation, $v(t,x)$ can never hit zero \cite{mueller-1998}.
 \begin{proposition}[Lemma 2.7 of \cite{mueller-1998}] \label{prop:positivity}
 	For any fixed $T>0$,
 	\begin{equation}
 		\lim_{\e \to 0} \Pro(\tau^{\inf}_\e <T \wedge \tau^\infty_\infty) = 0.
 	\end{equation}
 \end{proposition}
 \begin{proof}
 	The proof is the same as in \cite{mueller-1998}. Our assumptions look a little bit different because $b$ is nonzero, but this is not a problem. We assumed that $b$ is continuous near $0$. Therefore, there exists $\e_0>0$ such that for all $|u|< \e_0$,
 	\[u^{-\alpha} + b(u) \geq \frac{1}{2}u^{-\alpha}.\]
 	The result follows by Mueller's proof and comparison principle.
 \end{proof}

\section{Analysis of the $L^1$ norm} \label{S:L1}

Let $v(t,x)$ denote the local mild solution to \eqref{eq:v}.
As in the previous section, let 
\begin{equation} \label{eq:tau-infty}
	\tau^\infty_n = \inf\{t>0: |v(t)|_{L^\infty} \geq n\},
\end{equation}
and
\begin{equation} \label{eq:tau-inf}
	\tau^{\inf}_\e = \inf\left\{t>0: \inf_{x \in D} v(t,x) < \e\right\}.
\end{equation}
We introduce the $L^1$ stopping times for $M>1$
\begin{equation} \label{eq:tau-L1}
	\tau^1_M = \inf\{t>0: |v(t)|_{L^1} \geq M\}.
\end{equation}

The goal of this section is to prove that the $L^1$ norm of $v(t,x)$ stays finite. Then, in later sections we can use the finiteness of the $L^1$ norm to prove that the $L^\infty$ norms cannot explode.
Let 
\begin{equation}
	I_{n,\e}(t) := \int v(t \wedge \tau^\infty_n \wedge \tau^{\inf}_\e,x)dx.
\end{equation}
Because the spatial integral of the periodic heat kernel $\int G(t,x-y)dx\equiv 1$ for all $t>0$ and $y \in [-\pi,\pi]$, $I_{n,\e}(t)$ is a nonnegative semimartingale solving
\begin{align} \label{eq:I-semimartingale}
	I_{n,\e}(t)&= I_{n,\e}(0) + \int_0^{t \wedge \tau^\infty_n \wedge \tau^{\inf}_\e} \int  b(v(s,x))dxds \nonumber\\
	&+\int_0^{t \wedge \tau^\infty_n \wedge \tau^{\inf}_\e} \int (v(s,x))^{-\alpha}dxds+ \int_0^{t \wedge \tau^\infty_n \wedge \tau^{\inf}_\e}\int \sigma(v(s,x))W(dxds).
\end{align}

We divide the analysis of the $L^1$ norm into two cases, depending on whether we assume Assumption \ref{assum}(a) or (b).

\subsection{$L^1$ norm stays finite -- Assumption \ref{assum}(a)} \label{eq:SS-polynomial-case}
The main result of this subsection is the following.
\begin{lemma} \label{lem:L1-finite-assum-a}
	Assume $b$ and $\sigma$ satisfy Assumption \ref{assum}(a). Then there exists $C>0$ such that for any $T>0$, $M>0$, and  $\e \in (0,1)$,
	\begin{equation}
		\Pro \left( \sup_{t \in [0,T\wedge \tau^\infty_\infty \wedge \tau^{\inf}_\e)} \int u(t,x)dx > M \right)
		\leq \frac{\log\left(1+ |u(0)|_{L^1}\right) + (C+2\pi\e^{-\alpha})T}{\log(1 + M)}.
	\end{equation}
	In particular, this proves that the $L^1$ norm stays finite.
\end{lemma}

First we prove some preliminary results.
Let $h$ be from Assumption \ref{assum}(a) and define $f(u) = uh(u)$ and $g(u) = uh^{-1}(u)$. The next lemma proves that both $f$ and $g$ are positive, increasing, and convex. 

\begin{lemma}\label{lem:convexity-f-g}
	Assumption \ref{assum}(a) implies that both $f(u)  = uh(u)$ and $g(u)=uh^{-1}(u)$ are convex. 
\end{lemma}
\begin{proof}
	The convexity of $f$ is straightforward because  $f''(u) = 2h'(u) + uh''(u)\geq 0$ because $h$ is increasing and convex.
	
	To see the convexity of $g$, calculate
	\[g''(u) = \frac{2}{h'(h^{-1}(u))} - \frac{u h''(h^{-1}(u))}{(h'(h^{-1}(u)))^3}.\]
	Then
	\[g''(h(u)) = \frac{1}{h'(u)} \left(2 - \frac{h(u)h''(u)}{(h'(u))^2} \right) .\]
	This expression is nonnegative by assumption \eqref{eq:convexity-case-a}. Therefore $g$ is convex.
\end{proof}

\begin{lemma}
	The functions $f$ and $g$ have the property that $f^{-1}(u)g^{-1}(u) = u$ for all $u>h(0)$.
\end{lemma}
\begin{proof}
	By the definitions,
	$h^{-1}(u) = \frac{g(u)}{u}$. Therefore,
	\[f\left(\frac{g(u)}{u}\right) = \left(\frac{g(u)}{u}\right)u = g(u).\]
	Therefore,
	\[\frac{g(u)}{u} = f^{-1}(g(u)).\]
	Substituting $v=g(u)$, this proves
	\[v = f^{-1}(v)g^{-1}(v).\]
\end{proof}

The next lemma is a consequence of Jensen's inequality.
\begin{lemma} \label{lem:Jensen}
	For any nonnegative function $v :[-\pi,\pi] \to [0,+\infty)$, for which $\int v(x)h(v(x))dx<+\infty$, 
	\begin{equation} \label{eq:Jensen}
		\left(\int_{-\pi}^\pi h(v(x))dx \right)\leq   \frac{2\pi\int_{-\pi}^\pi (v(x) + \frac{1}{2\pi})h(v(x))dx}{1+ \int_{-\pi}^\pi v(x)dx} .
	\end{equation}
\end{lemma}

\begin{proof}
	By Lemma \ref{lem:convexity-f-g}, $f(u)=uh(u)$ and $g(u)=uh^{-1}(u)$ are both convex functions and they satisfy $f(v) = vh(v)$ and $g(h(v))=vh(v)$. By Jensen's inequality
	\begin{align}
		&\frac{1}{2\pi}\int_{-\pi}^\pi v(x)dx \leq f^{-1} \left(\frac{1}{2\pi} \int_{-\pi}^\pi v(x)h(v(x))dx\right) \\ 
		&\frac{1}{2\pi}\int_{-\pi}^\pi h(v(x))dx \leq g^{-1} \left(\frac{1}{2\pi} \int_{-\pi}^\pi v(x)h(v(x))dx\right). \label{eq:hv-bound}
	\end{align}
	Multiply these two equations and use the fact that $f^{-1}(u)g^{-1}(u)=u$ to see that
	\begin{equation} \label{eq:product-of-integrals}
		\left(\int_{-\pi}^\pi v(x)dx \right)\left(\int_{-\pi}^\pi h(v(x))dx\right) \leq 2\pi \int_{-\pi}^\pi v(x)h(v(x))dx.
	\end{equation}
	Add $\int h(v(x))dx$ to both sides
	\begin{equation} \label{eq:product-of-integrals}
		\left(1 + \int_{-\pi}^\pi v(x)dx \right)\left(\int_{-\pi}^\pi h(v(x))dx\right) \leq 2\pi \int_{-\pi}^\pi\left(\frac{1}{2\pi}+ v(x)\right)h(v(x))dx.
	\end{equation}
	This proves the claim.
\end{proof}

Now we can prove Lemma \ref{lem:L1-finite-assum-a}.
\begin{proof}[Proof of Lemma \ref{lem:L1-finite-assum-a}]
	Let $I_{n,\e}(t)$ be the semimartingale solving \eqref{eq:I-semimartingale}. By Ito formula
	\begin{align*}
		\log\left(1 + I_{n,\e}(t) \right) = &\log\left(1 + I_{n,\e}(0)\right)
		+\int_0^{t \wedge \tau^\infty_n \wedge \tau^{\inf}_\e} \int \frac{b(v(s,y))}{1 + I_{n,\e}(s)}dyds \\
		&+\int_0^{t \wedge \tau^\infty_n \wedge \tau^{\inf}_\e} \int \frac{(v(s,y))^{-\alpha}}{1 + I_{n,\e}(s)}dyds\\
		&-\frac{1}{2}\int_0^{t \wedge \tau^\infty_n \wedge \tau^{\inf}_\e} \int \frac{\sigma^2(v(s,y))}{(1 + I_{n,\e}(s))^2}dyds\\
		&+N(t) \nonumber\\
		&=:\log(1 + I_{n,\e}(0)) + B_{n,\e}(t) + A_{n,\e}(t) - S_{n,\e}(t) + N(t)
	\end{align*}
	where $N(t)$ is a martingale.
	
	Because of the definition of the $\tau^{\inf}_\e$ stopping time, and the fact that $I_{n,\e}(t)\geq 0$,
	\begin{equation}
		A_{n,\e}(t)\leq 2\pi t\e^{-\alpha}.
	\end{equation}
	
	By Assumption \ref{assum}(a) and Lemma \ref{lem:Jensen}
	\begin{align}
		&B_{n,\e}(t) = \int_0^{t \wedge \tau^\infty_n \wedge \tau^{\inf}_\e}\int \frac{b(v(s,y))}{1 + I_{n,\e}(s)}dy ds \\
		&\leq \int_0^{t \wedge \tau^\infty_n \wedge \tau^{\inf}_\e}\int \frac{h(v(s,y)) + \theta(1 + v(s,y))}{(1 + I_{n,\e}(s))}dyds\\
		& \leq  2\pi \int_0^{t \wedge \tau^\infty_n \wedge \tau^{\inf}_\e}\int \frac{(\frac{1}{2\pi}+v(s,y))h(v(s,y))}{(1 + I_{n,\e}(s))^2}dyds + Ct\\
		&\leq  \frac{1}{2}\int_0^{t \wedge \tau^\infty_n \wedge \tau^{\inf}_\e}\int \frac{\sigma^2(v(s,y))}{(1 + I_{n,\e}(s))^2}dyds + Ct.  
	\end{align}
	In the last line, we used the assumption \eqref{eq:sigma-upper-lower-case-a} that $(\frac{1}{2\pi}+u)h(u)\leq \frac{1}{4\pi}\sigma^2(u) $.
	Therefore, 
	\begin{equation}
		B_{n,\e}(t) - S_{n,\e}(t) \leq Ct
	\end{equation}
	
	These estimates imply that
	\begin{equation}
		\log(1 + I_{n,\e}(t)) \leq \log(1 + I_{n,\e}(0) ) + (C+2\pi\e^{-\alpha})t + N(t).
	\end{equation}
	By Doob's submartingale inequality, for any $T>0$, $M>0$,
	\begin{align}
		&\Pro\left(\sup_{t \in [0,T]} I_{n,\e}(t) >M\right)
		=\Pro\left(\sup_{t \in [0,T]} \log(1+I_{n,\e}(t)) >\log(1 + M)\right)\nonumber\\
		&
		\leq \frac{\log(1 + |u(0)|_{L^1}) +(C + 2\pi\e^{-\alpha})T}{\log(1+M)}.
	\end{align}
	Where $C$ does not depend on $n$, $\e$, $T$, or $|u(0)|_{L^1}$. Importantly, the right-hand side is independent of $n$, leading to our claim.
\end{proof}

\subsection{$L^1$ norm stays finite -- Assumption \ref{assum}(b)} \label{SS:subsec-Osgood-case}
In this section, we are assuming Assumption \ref{assum} holds. 
We define a new function
\begin{equation} \label{eq:g-def}
	g(u) := e^{-u} (h(e^u)-h(0)),
\end{equation} 
which, without confusion, is completely different than the $g$ defined in the previous subsection.
In this way, we can write 
\begin{equation} \label{eq:h-in-terms-of-g}
	h(u)=h(0)+ug(\log(u)).
\end{equation} 
Furthermore, because $h$ is positive, increasing, and convex, we know that $g(u)$ is increasing. By assumption \eqref{eq:h-osgood-case-b}, it is clear that
\begin{equation} \label{eq:g-osgood}
	\int_1^\infty \frac{1}{g(u)}du = \int_1^\infty \frac{e^u}{h(e^u) -h(0)}du
	=\int_e^\infty \frac{1}{h(u)-h(0)} = + \infty.
\end{equation}
Assumption \eqref{eq:h-growth-case-b} and \eqref{eq:h-in-terms-of-g} imply that
\begin{equation} \label{eq:g-growth}
	\limsup_{u \to +\infty} \frac{g(\log(u^2))}{ u^{2\gamma -1}} =0.
\end{equation}

\begin{lemma} \label{lem:int-b-bound}
	Let $I_{n,\e}(t)$ denote the solution to \eqref{eq:I-semimartingale}. There exists $C>0$ such that for any  $s \in [0,\tau^\infty_n \wedge \tau^{\inf}_\e]$,
	\begin{align}
		\frac{\int b(v(s,x))dx}{(1 + I_{n,\e}(s))}  \leq C+  &g\left(\frac{2}{2\gamma -1 }\log(1 + I_{n,\e}(s))\right)  \nonumber\\
		&+   \int \frac{v(s,x)}{(1 + I_{n,\e}(s))} g \left(2\log \left(\frac{v(s,x)}{(1 + I_{n,\e}(s))^{\frac{1}{2\gamma-1}}}\right)\right)dx.
	\end{align}
	where $\gamma$ is from Assumption \ref{assum}(b).
\end{lemma}

\begin{proof}
	By Asumption \eqref{eq:b-upper-case-b} and the definition of $g$ \eqref{eq:g-def},
	$$b(v(s,x)) \leq h(0) + v(s,x) g( \log(v(s,x))).$$ 
	By the properties of logarithms,
	\[\log(v(s,x)) \leq \log \left(\frac{v(s,x)}{(1 + I_{n,\e}(s))^{\frac{1}{2\gamma -1}}}\right) + \frac{1}{2\gamma -1} \log \left(1 + I_{n,\e}(s)\right).\]
	Because $g$ is positive and increasing, for any $a,b \in \mathbb{R}$, $g(a+b) \leq \max\{g(2a), g(2b)\} \leq g(2a) + g(2b)$. Therefore,
	\[g(\log(v(s,x))) \leq  g \left(2\log \left(\frac{v(s,x)}{(1 + I_{n,\e}(s))^{\frac{1}{2\gamma -1}}}\right)\right) +  g \left(\frac{2}{2\gamma -1 }\log \left(1 + I_{n,\e}(s)\right)\right).\]
	
	Now we multiply this estimate by $v(t,x)$ and integrate in space and use the definition that $I_{n,\e}(s) = \int v(s,x)dx$ to observe that
	\begin{align}
		\int v(s,x)g(\log(v(s,x)))dx \leq 
		&\int v(s,x) g \left(2\log \left(\frac{v(s,x)}{(1 + I_{n,\e}(s))^{\frac{1}{2\gamma-1}}}\right)\right)dx \nonumber\\
		& + I_{n,\e}(s) g \left(\frac{2}{2\gamma -1}\log \left(1 + I_{n,\e}(s)\right)\right).
	\end{align}
	The result follows because Assumption \eqref{eq:b-upper-case-b} assumes that $$b(v(s,x)) \leq h(0) + v(s,x)g(v(s,x)).$$
\end{proof}

\begin{lemma}\label{lem:L1-finite-assum-b}
	 Assume Assumption \ref{assum}(b). Let 
	\[G(x) = \int_1^x \frac{1}{g\left(\frac{2}{2\gamma -1 }u\right)}du\]
	and notice that by \eqref{eq:g-osgood}, $\lim_{x \to \infty} G(x) = \infty$.
	There exists $C>0$  such that for any $T>0$, $n>0$, and $M>0$, 
	\begin{align}
		&\Pro \left( \sup_{t \in [0,T\wedge \tau^\infty_\infty \wedge \tau^{\inf}_\e)} \int u(t,x)dx > M \right) \nonumber\\
		&\leq \frac{G(\log(1 + I_{n,\e}(0))) + (C + 2\pi \e^{-\alpha})T}{G(\log(1 + M))}.
	\end{align}
\end{lemma}

\begin{proof}
Apply Ito formula with the nonnegative function $ \log(1 + I_{n,\e}(t))$.
\begin{align} \label{eq:log-Ito}
	\log&(1 + I_{n,\e}(t)) \nonumber\\
	= &\log(1 + I_{n,\e}(0))\nonumber\\
	&+ \int_0^{t \wedge \tau^\infty_n \wedge \tau^{\inf}_\e} \frac{\int(v(s,y)dy)^{-\alpha}}{1 + I_{n,\e}(s)}ds
	+ \int_0^{t \wedge \tau^\infty_n \wedge \tau^{\inf}_\e} \frac{\int b(v(s,x))dx}{1 + I_{n,\e}(s)}ds\nonumber\\
	&- \int_0^{t \wedge \tau^\infty_n \wedge \tau^{\inf}_\e} \frac{\int (\sigma(v(s,x)))^2dx }{2(1+ I_{n,\e}(s))^2}ds +N(t)\nonumber\\
	&=: \log(1 + I_{n,\e}(0))+ A_{n,\e}(t) + B_{n,\e}(t) - S_{n,\e}(t) + N(t)
\end{align}
where $N(t)$ is a martingale.

Because of the stopping time $\tau^{\inf}_\e$,
\begin{equation}
	A_{n,\e}(t) \leq 2\pi \e^{-\alpha}t.
\end{equation}
Apply Lemma \ref{lem:int-b-bound} to see that
\begin{align}
	&B_{n,\e}(t):= \int_0^{t \wedge \tau^\infty_n \wedge \tau^{\inf}_\e} \frac{\int b(v(s,x))dx}{(1 + I_{n,\e}(s))}ds\nonumber\\
	&\leq C(t \wedge \tau^\infty_n \wedge \tau^{\inf}_\e)+   \int_0^{t \wedge \tau^\infty_n \wedge \tau^{\inf}_\e}  \int \frac{v(s,x)}{(1 + I_{n,\e}(s))} g \left(2\log \left(\frac{v(s,x)}{(1 + I_{n,\e}(s))^{\frac{1}{2\gamma-1}}}\right)\right)dxds \nonumber\\
	&\qquad+   \int_0^{t \wedge \tau^\infty_n \wedge \tau^{\inf}_\e} g \left(\frac{2}{2\gamma-1}\log\left(1 + I_{n,\e}(s)\right)\right)ds.
\end{align}
In the last line, we bounded $\frac{I_{n,\e}(s)}{1 +I_{n,\e}(s)} \leq 1.$

Use the assumption \eqref{eq:sigma-upper-lower-case-b} to lower  bound
\begin{align}
	&S_{n,\e}(t) := \int_0^{t \wedge \tau^\infty_n \wedge \tau^{\inf}_\e} \frac{\int (\sigma(v(s,x)))^2dx }{2(1 + I_{n,\e}(s))^2}ds \nonumber\\
	&\geq \int_0^{t \wedge \tau^\infty_n \wedge \tau^{\inf}_\e} \frac{\int c^2 (v(s,x))^{2\gamma}dx }{2(1 + I_{n,\e}(s))^2}ds \nonumber\\
	&\geq \frac{c^2}{2}\int_0^{t \wedge \tau^\infty_\infty \wedge \tau^{\inf}_\e} \int \left(\frac{v(s,x)}{1 + I_{n,\e}(s)}\right)\left(\frac{v(s,x)}{(1 + I_{n,\e}(s))^{\frac{1}{2\gamma-1}}} \right)^{2\gamma -1}dxds.
\end{align}

Therefore,
\begin{align}
	B_{n,\e}(t)& - S_{n,\e}(t) 
	\leq C(t \wedge \tau^\infty_n \wedge \tau^{\inf}_\e)   \nonumber\\
	&+\int_0^{t \wedge \tau^\infty_n \wedge \tau^{\inf}_\e} g \left(\frac{2}{2\gamma-1}\log\left(1 + I_{n,\e}(s)\right)\right)ds\nonumber\\
	&+ \int_0^{t \wedge \tau^\infty_n \wedge \tau^{\inf}_\e} \int \left( \frac{v(s,x)}{1 + I_{n,\e}(s)} \right) \Bigg(g \left(2\log \left(\frac{v(s,x)}{(1 + I_{n,\e}(s))^{\frac{1}{2\gamma-1}}}\right)\right) \nonumber\\
	&\hspace{5cm}- \frac{c^2}{2}\left(\frac{v(s,x)}{(1 + I_{n,\e}(s))^{\frac{1}{2\gamma-1}}} \right)^{2\gamma -1}\Bigg)dyds.
\end{align}

Now we use the bound \eqref{eq:g-growth} to see that 
$$\sup_{u>0} (g(2\log(u)) - \frac{c^2}{2}u^{2\gamma-1})<+\infty,$$ 
with $\frac{v(s,x)}{(1 + I_{n,\e}(s))^{2\gamma-1}}$ replacing $u$, to conclude that there exist $C>0$  such that

\begin{align}
	B_{n,\e}(t) - S_{n,\e}(t) \leq &C({t \wedge \tau^\infty_n \wedge \tau^{\inf}_\e}) \nonumber\\
	&+  \int_0^{t \wedge \tau^\infty_n \wedge \tau^{\inf}_\e} g \left(\frac{2}{2\gamma -1}\log\left(1 + I_{n,\e}(s)\right)\right)ds \nonumber\\
	&+C \int_0^{t \wedge \tau^\infty_n \wedge \tau^{\inf}_\e} \int \left(\frac{v(s,x)}{1 + I_{n,\e}(s)} \right)dx.
\end{align}

We can simplify this expression because $\int \frac{v(t,x)}{(1 + I_{n,\e}(s))}dx \leq 1.$

Therefore, 
\begin{align}
	\log(1 + I_{n,\e}(t)) = &\log(1 + I_{n,\e}(0)) + (C + 2\pi \e^{-\alpha})t\nonumber\\
	&+ C \int_0^{t \wedge \tau^\infty_n \wedge \tau^{\inf}_\e} g\left(\frac{2}{2\gamma-1}\log(1 + I_{n,\e}(s))\right)ds +N(t).
\end{align}

Applying Ito formula with $G(x) = \int_1^x \frac{1}{g(\frac{2}{2\gamma-1}u)}du$, and noticing that $G''(u)\leq0$ because $\frac{1}{g(u)}$ is nonincreasing,
\begin{align}
	G(\log(1 + I_{n,\e}(t))) \leq G(\log(1 + I_{n,\e}(0))) +(C + 2\pi \e^{-\alpha})t + \tilde{N}(t),
\end{align}
where $\tilde{N}(t)$ is a martingale.
$G$ is nonnegative and increasing, so Doob's submartingale inequality implies that for any $M>0$,
\begin{align}
	&\Pro \left(\sup_{t \in [0,T ]} I_{n,\e}(t) > M \right)\nonumber\\
	&=\Pro \left(\sup_{t \in [0, T ]} G(\log(1 + I_{n,\e}(t))) > G(\log(1 +M)) \right)\nonumber\\
	&\leq \frac{\E G(\log(1 + I_{n,\e}(T)))}{G(\log(1 + M))}\nonumber\\
	&\leq \frac{G(\log(1 + |u(0)|_{L^1})) + (C+ 2\pi\e^{-\alpha})T}{G(\log(1 + M))}
\end{align}

The above bound is independent of $n$, proving the claim.
\end{proof}

\section{$L^\infty$ norm} \label{S:Linfinity}

Let $\tau^{\inf}_\e$ and $\tau^1_M$ be defined as in \eqref{eq:tau-inf}--\eqref{eq:tau-L1}.
We now define a sequence of stopping times that keep track of when the $L^\infty$ norm triples or falls by one-third. The solution will explode if and only if the $L^\infty$ norm triples an infinite number of times, and we can prove that this infinite tripling cannot happen before the time $\tau^1_M\wedge \tau^{\inf}_\e$.

Given any fixed $M>0$ and $\e>0$, define
\begin{equation}
	\rho_0 = \inf\{t\in [0,\tau^1_M \wedge \tau^{\inf}_\e]: |u(t)|_{L^\infty} = 3^m \text{ for some } m \in \mathbb{N} \},
\end{equation}
If $|u(\rho_n)|_{L^\infty} = 3^0$, then
\begin{equation}
	\rho_{n+1}= \inf\{t \in [\rho_n, \tau^1_M\wedge \tau^{\inf}_\e]: |u(t)|_{L^\infty} = 3^{1} \}
\end{equation}
and if $|u(\rho_n)|_{L^\infty} = 3^m$ for $m\geq 1$, then
\begin{equation}
	\rho_{n+1}= \inf\{t \in [\rho_n, \tau^1_M\wedge \tau^{\inf}_\e]: |u(t)|_{L^\infty} = 3^{m+1} \text{ or } |u(t)|_{L^\infty} = 3^{m-1}\}
\end{equation}
For all of these stopping time definitions, we use the convention that $\rho_n = \tau^1_M \wedge \tau^{\inf}_\e$ if the $L^\infty$ norm does not triple or fall by one third after $\rho_{n-1}$.

We recall the following moment bound on the supremum norm of a stochastic convolution from \cite{s-2024-aop}

\begin{proposition}[Theorem 1.2 of \cite{s-2024-aop}] \label{prop:moment}
	Define the stochastic convolution for an adapted $\varphi: [0,T]\times [-\pi,\pi] \to \mathbb{R}$ by
	\begin{equation}
		Z^\varphi(t,x) = \int_0^t \int G(t-s,x-y) \varphi(s,y)W(dyds).
	\end{equation}
	For any $p>6$, there exists $C_p>0$ such that for any $T \in (0,1)$ and any adapted $\varphi$,
	
	\begin{equation}
		\E\sup_{t \in [0,T]}\sup_{x \in D} |Z^\varphi(t,x)|^p \leq C_p T^{\frac{p}{4}-\frac{3}{2}}\E \int_0^T \int |\varphi(t,x)|^pdxdt.
	\end{equation}
\end{proposition}

The next result is the fundamental lemma that allows us to prove non-explosion.

\begin{lemma} \label{lem:tripling-prob}
	There exists $m_0>0$ and $C>0$, depending on $\e$ and $M$, such that for any $m> m_0$ and  any $n \in \mathbb{N}$, the probability of tripling satisfies
	\begin{align}
		&\Pro \left(|v(\rho_{n+1})|_{L^\infty} =3|v(\rho_n)|_{L^\infty} \text{ and } |v(\rho_n)|_{L^\infty} \geq 3^{m_0} \right) \nonumber\\
		&\leq C \E\int_{\rho_n}^{\rho_{n+1}} \int |\sigma(v(s,y))|^2dyds.
	\end{align}
\end{lemma}

\begin{proof}
	Starting at time $\rho_n$, for $t \in [0, \rho_{n+1}-\rho_n]$, the local mild solution satisfies
	\begin{equation}
		v(\rho_n+t,x) = \tilde{v}_n(t,x) 
	\end{equation}
	where 
	\begin{equation}
		\tilde{v}_n(t,x):=\int G(t,x-y)v(\rho_n,y)dy  + I^1_n(t,x) + I_n^2(t,x) +  Z_n(t,x)
	\end{equation}
	where
	\begin{equation}
		I_n^1(t,x) = \int_{\rho_n}^{(\rho_n + t) \wedge \rho_{n+1}} \int G(t-s,x-y)b(v(s,y))dyds, 
	\end{equation}
	\begin{equation}
		I_n^2(t,x) =  \int_{\rho_n}^{(\rho_n + t) \wedge \rho_{n+1}}\int G(t-s,x-y) (v(s,y))^{-\alpha}dyds,
	\end{equation}
	and $Z_n$ is the stochastic convolution
	\begin{equation}
		Z_n(t,x) = \int_{\rho_n}^{(\rho_n + t) \wedge \rho_{n+1}} \int G(t-s,x-y) \sigma(v(s,y))W(dyds).
	\end{equation}
	These definitions of $\tilde{v}_n(t,x)$, $I_n^1(t,x)$, $I_n^2(t,x)$, and $Z_n(t,x)$  are convenient because they well-defined for all $t>0$, and $x \in D$, while $v$ without stopping time could potentially explode. The definition satisfies $\tilde{v}_n(t,x) = v(\rho_n + t,x)$ for all $t \in [0, \rho_{n+1} - \rho_n]$. We will show that the $L^\infty$ norm of $\tilde{v}_n$ falls by a third before it triples with high probability. Because the original $v(\rho_n + t,x)$ matches until this tripling or falling by one-third occurs, this will prove that $v(\rho_{n+1}+t,x)$ also falls by one-third before tripling with high probability. 
	
	The heat kernel satisfies the property $|G(t,x)| \leq C t^{-\frac{1}{2}}$ for $t \in (0,1)$. Let $T_m = C^{-2}M^{-2}3^{-2(m-2)}$. Then, conditioning on $|u(\rho_n)|_{L^\infty}=3^m$, and using the fact that $\rho_n \leq \tau^1_M$ by definition, the linear term satisfies
	\begin{equation}
		\int G(T_m, x-y) v(\rho_n,y)dy \leq CMT_m^{-\frac{1}{2}} \leq 3^{m-2}.
	\end{equation}
	On the other hand, because $|G(T_m,\cdot)|_{L^1}=1$,
	for any $t >0$,
	\begin{equation}
		\int G(t,x-y)v(\rho_n,y)dy \leq 3^m.
	\end{equation}
	Because of the definition that $\rho_{n+1}\leq \tau^{\inf}_\e$,
	\begin{equation}
		\sup_{t \leq T_m} \sup_{x\in D} I_n^2(t,x) \leq \e^{-\alpha}T_m \leq C \e^{-\alpha}3^{-2m}.
	\end{equation}
	Conditioning on $|u(\rho_n)|_{L^\infty} =3^m$, and applying Assumption \ref{assum},
	\begin{equation}
		\sup_{t \leq T_m} \sup_{x \in D} I_n^1(t,x) \leq T_m (\theta(1 + 3^{m+1})+h(3^{m+1})).
	\end{equation}
	Under either Assumption \ref{assum}(a) or (b), the growth rate of $h(u)\leq Cu^2$ for some $C$. Therefore,
	\begin{equation}
		\sup_{t \leq T_m} \sup_{x \in D}I_n^1(t,x) \leq C3^{-2m}3^{2m} \leq C.
	\end{equation}
	Therefore, we can choose $m_0$, depending on $\e$, large enough so that for all $m\geq m_0$,
	\begin{equation}
		I_n^1(t,x) + I_n^2(t,x) \leq C \e^{-\alpha }3^{-2m}  +C \leq 3^{m-2}.
	\end{equation}
	
	These estimates lead to the estimates that for $t \leq T_m$
	\begin{align}
		&|\tilde{v}_n(t)|_{L^\infty} \leq 3^m + 3^{m-2} + \sup_{t \leq T_m}\sup_{x \in D} |Z_n(t,x)|
	\end{align}
	and the endpoint estimate
	\begin{equation}
		|\tilde{v}_n(T_m)|_{L^\infty} \leq 3^{m-2} + 3^{m-2} + \sup_{t \leq T_m}\sup_{x \in D} |Z_n(t,x)|.
	\end{equation}
	
	If $|Z_n(t,x)|_{L^\infty}\leq 3^{m-2}$ for all $t \leq T_m$ and $x \in D$, then $|\tilde{v}_n(t)|_{L^\infty} \leq 3^{m} + 3^{m-2} + 3^{m-2}< 3^{m+1}$ and $|\tilde{v}_n(T_m)|_{L^\infty} \leq 3^{m-2} + 3^{m-2} + 3^{m-1} \leq 3^{m-1}$. In this case, therefore, $|\tilde{v}_n(t,x)|_{L^\infty} = |v(\rho_n+t,x)|_{L^\infty}$ will fall by a third before it can triple. This means if $|v(\rho_{n+1})|_{L^\infty}$ is going to reach the level $3^{m+1}$ before falling to $3^{m-1}$, then 
	\begin{align}
		&\Pro \left(|v(\rho_{n+1})|_{L^\infty} = 3^{m+1} \Big| |v(\rho_n)|_{L^\infty} = 3^m\right) \nonumber\\
		&\leq 
		\Pro\left(\sup_{t \leq T_m} \sup_{x \in D} |Z_n(t,x)|>3^{m-2}\Big| |v(\rho_n)|_{L^\infty} = 3^m\right).
	\end{align}
	
	By Chebyshev's inequality and the moment bound Lemma \ref{prop:moment},
	\begin{align}
		&\Pro\left(\sup_{t \leq T_m} \sup_{x \in D} |Z_n(t,x)| > 3^{m-2} \Big| |v(\rho_n)|_{L^\infty} = 3^m\right)\nonumber\\
		&\leq C3^{-mp} \E \left(\sup_{t \leq T_m} \sup_{x \in D} |Z_n(t,x)|^p \Big| |v(\rho_n)|_{L^\infty} =3^m\right)\nonumber\\
		&\leq C 3^{-\frac{3mp}{2}+3m}\E\left( \int_{\rho_n}^{\rho_{n+1}} \int |\sigma(v(s,y))|^pdyds\Big| |v(\rho_n)|_{L^\infty} =3^m\right).
	\end{align}
	By Assumption \ref{assum}, $|\sigma(v(s,y))| \leq C3^{\frac{3m}{2}}$. Therefore, factoring out the supremum norm of $|\sigma(v(s,y))|^{p-2}$ in the above expression,
	\begin{align}
		&\Pro\left(\sup_{t \leq T_m} \sup_{x \in D} |Z_n(t,x)| > 3^{m-2} \Big| |v(\rho_n)|_{L^\infty} = 3^m\right)\nonumber\\
		&\leq C 3^{-\frac{3mp}{2} +mp + \frac{3m(p-2)}{2}} \E \left(\int_{\rho_n}^{\rho_{n+1}} \int |\sigma(u(s,y))|^2dyds \Big| |v(\rho_n)|_{L^\infty} =3^m\right).
	\end{align} 
	All of $m$ terms cancel out. The leading constant $C$ depends on $M$, $\e$, and $p$, which all have been fixed, but $C$ is uniform for $m>m_0$ and $n$. 
	By the law of total probability,
	\begin{align}
		&\Pro \left(|v(\rho_{n+1})|_{L^\infty} = 3 |v(\rho_n)|_{L^\infty} \text{ and } |v(\rho_n)|_{L^\infty} \geq 3^{m_0}\right) \nonumber\\
		&= \sum_{m=m_0}^\infty \Pro \left(|v(\rho_{n+1})|_{L^\infty} = 3^{m+1} | |v(\rho_n)|_{L^\infty} = 3^m\right) \Pro \left(|v(\rho_n)|_{L^\infty} = 3^m\right) \ \nonumber\\
		&\leq C  \E \int_{\rho_n}^{\rho_{n+1}} \int|\sigma(u(s,y))|^2dyds.
	\end{align}
\end{proof}

Next, we use the previously developed bounds to prove that the explosion time is larger than $\tau^{\inf}_\e \wedge \tau^1_M$ .

\begin{lemma} \label{lem:explosion-after-tau-inf-tau-1}
	For any $\e>0$ and $M>0$, 
	\begin{equation}
		\Pro\left(\tau^\infty_\infty> \tau^1_M \wedge \tau^{\inf}_\e\right)=1.
	\end{equation}
\end{lemma}

\begin{proof}
	The estimate from Lemma \ref{lem:tripling-prob} is uniform with respect to $n>0$. 
	
	Add these up with respect to $n$, remembering that $\rho_n\leq \rho_{n+1} \leq \tau^1_M \wedge \tau^{\inf}_\e$, to see that
	\begin{align}
		&\sum_{n=1}^\infty\Pro \left(|v(\rho_{n+1})|_{L^\infty} = 3 |v(\rho_n)|_{L^\infty} \text{ and } |v(\rho_n)|_{L^\infty} \geq 3^{m_0}\right) \nonumber\\
		&\leq  C  \E \int_{0}^{\tau^1_M \wedge \tau^{\inf}_\e}\int |\sigma(u(s,y))|^2dyds.
		\end{align}
		The right-hand side is proportional to the the quadratic variation of the $L^1$ norm, which is bounded by $M^2$ because of the stopping time
		\begin{equation}
			M^2 \geq \E|u(\tau^1_M \wedge \tau^{\inf}_\e)|_{L^1}^2 = |u(0)|_{L^1}^2 +\E \int_0^{\tau^1_M \wedge \tau^{\inf}_\e} \int |\sigma(u(s,y))|^2 dyds.
		\end{equation}
		
		The Borel-Cantelli Lemma guarantees that the $L^\infty$ norm can only triple a finite number of times before time $\tau^1_M \wedge \tau^{\inf}_\e$, when $m$ is large. This implies  that $\tau^\infty_\infty>\tau^1_M \wedge \tau^{\inf}_\e$ with probability one.
\end{proof}

\begin{proof}[Proof of Theorem \ref{thm:main}]
	We proved in Lemma \ref{lem:explosion-after-tau-inf-tau-1} that for any fixed $\e>0$ and $M>0$, 
	\begin{equation}
		\Pro(\tau^\infty_\infty > \tau^1_M \wedge \tau^{\inf}_\e).
	\end{equation}
	We can complete the proof that $v(t,x)$ cannot explode by proving that $\tau^1_M$ and $\tau^{\inf}_\e$ converge to $\infty$ as $\e \to 0$ and $M \to \infty$.
	
	Proposition \ref{prop:positivity} proves that  for arbitrary $\eta>0$ and $T>0$, there exists $\e_0= \e_0(T,\eta)>0$ such that
	\begin{equation}
		\Pro\left( \tau^{\inf}_{\e_0} \leq T \wedge \tau^\infty_\infty\right)< \frac{\eta}{2}.
	\end{equation}
	Then Lemma \ref{lem:L1-finite-assum-a} or Lemma \ref{lem:L1-finite-assum-b}, depending on whether we assume Assumption \ref{assum}(a) or (b), guarantees that there exists $M_0= M_0(T,\eta,\e_0)$ such that 
	\begin{align}
		& \Pro \left(\tau^1_{M_0} \leq T \wedge \tau^\infty_\infty \wedge \tau^{\inf}_{\e_0}\right)
		=\Pro \left( \sup_{t \in [0,T\wedge \tau^\infty_\infty \wedge \tau^{\inf}_\e)} \int u(t,x)dx > M_0 \right)
		<\frac{\eta}{2}.
	\end{align}
	Then because $\tau^\infty_\infty> \tau^1_{M_0} \wedge \tau^{\inf}_{\e_0}$, it follows that
	\begin{equation}
		\Pro\left(\tau^\infty_\infty \leq T\right) \leq \Pro\left( \tau^{\inf}_{\e_0} \leq T \wedge \tau^\infty_\infty\right) + \Pro \left(\tau^1_{M_0} \leq T \wedge \tau^\infty_\infty \wedge \tau^{\inf}_{\e_0}\right) < \eta.
	\end{equation}
	Because $\eta$ and $T$ were arbitrary, this proves that $\Pro(\tau^\infty_\infty = \infty)=1$ and the positive solution $v(t,x)$ defined in \eqref{eq:v} cannot explode in finite time. The argument for $v_-(t,x)$ defined in \eqref{eq:v-} is identical. Finally, by the comparison principle Proposition \ref{prop:comparison}, $-v_-(t,x) \leq u(t,x) \leq v(t,x)$ for all $t>0$ and $x \in [-\pi,\pi]$. Therefore, the local mild solution to $u(t,x)$ cannot explode.
	
\end{proof}

\bibliographystyle{./abbrv}
\bibliography{superlinear}

\begin{bibdiv}
\begin{biblist}

\bib{av-2023}{article}{
      author={Agresti, Antonio},
      author={Veraar, Mark},
       title={Reaction-diffusion equations with transport noise and critical
  superlinear diffusion: local well-posedness and positivity},
        date={2023},
        ISSN={0022-0396,1090-2732},
     journal={J. Differential Equations},
      volume={368},
       pages={247\ndash 300},
         url={https://doi.org/10.1016/j.jde.2023.05.038},
      review={\MR{4598499}},
}

\bib{bezdek-2018}{article}{
      author={Bezdek, Pavel},
       title={Existence and blow-up of solutions to the fractional stochastic
  heat equations},
        date={2018},
        ISSN={2194-0401,2194-041X},
     journal={Stoch. Partial Differ. Equ. Anal. Comput.},
      volume={6},
      number={1},
       pages={73\ndash 108},
         url={https://doi.org/10.1007/s40072-017-0103-8},
      review={\MR{3768995}},
}

\bib{ch-2023}{article}{
      author={Chen, Le},
      author={Huang, Jingyu},
       title={Superlinear stochastic heat equation on {$\Bbb{R}^d$}},
        date={2023},
        ISSN={0002-9939,1088-6826},
     journal={Proc. Amer. Math. Soc.},
      volume={151},
      number={9},
       pages={4063\ndash 4078},
         url={https://doi.org/10.1090/proc/16436},
      review={\MR{4607649}},
}

\bib{dpz}{book}{
      author={Da~Prato, Giuseppe},
      author={Zabczyk, Jerzy},
       title={Stochastic equations in infinite dimensions},
     edition={Second},
      series={Encyclopedia of Mathematics and its Applications},
   publisher={Cambridge University Press, Cambridge},
        date={2014},
      volume={152},
        ISBN={978-1-107-05584-1},
         url={https://doi.org/10.1017/CBO9781107295513},
      review={\MR{3236753}},
}

\bib{dalang-1999}{article}{
      author={Dalang, Robert~C.},
       title={Extending the martingale measure stochastic integral with
  applications to spatially homogeneous s.p.d.e.'s},
        date={1999},
        ISSN={1083-6489},
     journal={Electron. J. Probab.},
      volume={4},
       pages={no. 6, 29},
         url={https://doi.org/10.1214/EJP.v4-43},
      review={\MR{1684157}},
}

\bib{dkz-2019}{article}{
      author={Dalang, Robert~C.},
      author={Khoshnevisan, Davar},
      author={Zhang, Tusheng},
       title={Global solutions to stochastic reaction-diffusion equations with
  super-linear drift and multiplicative noise},
        date={2019},
        ISSN={0091-1798,2168-894X},
     journal={Ann. Probab.},
      volume={47},
      number={1},
       pages={519\ndash 559},
         url={https://doi.org/10.1214/18-AOP1270},
      review={\MR{3909975}},
}

\bib{bd-2002}{article}{
      author={de~Bouard, A.},
      author={Debussche, A.},
       title={On the effect of a noise on the solutions of the focusing
  supercritical nonlinear {S}chr\"{o}dinger equation},
        date={2002},
        ISSN={0178-8051,1432-2064},
     journal={Probab. Theory Related Fields},
      volume={123},
      number={1},
       pages={76\ndash 96},
         url={https://doi.org/10.1007/s004400100183},
      review={\MR{1906438}},
}

\bib{bg-2009}{article}{
      author={Fern\'{a}ndez~Bonder, Julian},
      author={Groisman, Pablo},
       title={Time-space white noise eliminates global solutions in
  reaction-diffusion equations},
        date={2009},
        ISSN={0167-2789,1872-8022},
     journal={Phys. D},
      volume={238},
      number={2},
       pages={209\ndash 215},
         url={https://doi.org/10.1016/j.physd.2008.09.005},
      review={\MR{2516340}},
}

\bib{fkn-2024}{article}{
      author={Foondun, Mohammud},
      author={Khoshnevisan, Davar},
      author={Nualart, Eulalia},
       title={Instantaneous everywhere-blowup of parabolic spdes},
        date={2024},
     journal={Probability Theory and Related Fields},
       pages={1\ndash 24},
}

\bib{fn-2021}{article}{
      author={Foondun, Mohammud},
      author={Nualart, Eulalia},
       title={The {O}sgood condition for stochastic partial differential
  equations},
        date={2021},
        ISSN={1350-7265,1573-9759},
     journal={Bernoulli},
      volume={27},
      number={1},
       pages={295\ndash 311},
         url={https://doi.org/10.3150/20-BEJ1240},
      review={\MR{4177371}},
}

\bib{kotelenez-1992}{article}{
      author={Kotelenez, Peter},
       title={Comparison methods for a class of function valued stochastic
  partial differential equations},
        date={1992},
        ISSN={0178-8051,1432-2064},
     journal={Probab. Theory Related Fields},
      volume={93},
      number={1},
       pages={1\ndash 19},
         url={https://doi.org/10.1007/BF01195385},
      review={\MR{1172936}},
}

\bib{krylov}{article}{
      author={Krylov, N.~V.},
       title={On {$L_p$}-theory of stochastic partial differential equations in
  the whole space},
        date={1996},
        ISSN={0036-1410},
     journal={SIAM J. Math. Anal.},
      volume={27},
      number={2},
       pages={313\ndash 340},
         url={https://doi.org/10.1137/S0036141094263317},
      review={\MR{1377477}},
}

\bib{lz-2022}{article}{
      author={Liang, Fei},
      author={Zhao, Shuangshuang},
       title={Global existence and finite time blow-up for a stochastic
  non-local reaction-diffusion equation},
        date={2022},
        ISSN={0393-0440,1879-1662},
     journal={J. Geom. Phys.},
      volume={178},
       pages={Paper No. 104577, 21},
         url={https://doi.org/10.1016/j.geomphys.2022.104577},
      review={\MR{4433058}},
}

\bib{ms-wave-2021}{article}{
      author={Millet, Annie},
      author={Sanz-Sol\'{e}, Marta},
       title={Global solutions to stochastic wave equations with superlinear
  coefficients},
        date={2021},
        ISSN={0304-4149,1879-209X},
     journal={Stochastic Process. Appl.},
      volume={139},
       pages={175\ndash 211},
         url={https://doi.org/10.1016/j.spa.2021.05.002},
      review={\MR{4264843}},
}

\bib{mueller-1991}{article}{
      author={Mueller, Carl},
       title={Long time existence for the heat equation with a noise term},
        date={1991},
        ISSN={0178-8051,1432-2064},
     journal={Probab. Theory Related Fields},
      volume={90},
      number={4},
       pages={505\ndash 517},
         url={https://doi.org/10.1007/BF01192141},
      review={\MR{1135557}},
}

\bib{mueller-1991-support}{article}{
      author={Mueller, Carl},
       title={On the support of solutions to the heat equation with noise},
        date={1991},
        ISSN={1045-1129},
     journal={Stochastics Stochastics Rep.},
      volume={37},
      number={4},
       pages={225\ndash 245},
         url={https://doi.org/10.1080/17442509108833738},
      review={\MR{1149348}},
}

\bib{mueller-1997}{article}{
      author={Mueller, Carl},
       title={Long time existence for the wave equation with a noise term},
        date={1997},
        ISSN={0091-1798,2168-894X},
     journal={Ann. Probab.},
      volume={25},
      number={1},
       pages={133\ndash 151},
         url={https://doi.org/10.1214/aop/1024404282},
      review={\MR{1428503}},
}

\bib{mueller-1998}{article}{
      author={Mueller, Carl},
       title={Long-time existence for signed solutions of the heat equation
  with a noise term},
        date={1998},
        ISSN={0178-8051,1432-2064},
     journal={Probab. Theory Related Fields},
      volume={110},
      number={1},
       pages={51\ndash 68},
         url={https://doi.org/10.1007/s004400050144},
      review={\MR{1602036}},
}

\bib{mueller-2000}{article}{
      author={Mueller, Carl},
       title={The critical parameter for the heat equation with a noise term to
  blow up in finite time},
        date={2000},
        ISSN={0091-1798,2168-894X},
     journal={Ann. Probab.},
      volume={28},
      number={4},
       pages={1735\ndash 1746},
         url={https://doi.org/10.1214/aop/1019160505},
      review={\MR{1813841}},
}

\bib{ms-1993}{article}{
      author={Mueller, Carl},
      author={Sowers, Richard},
       title={Blowup for the heat equation with a noise term},
        date={1993},
        ISSN={0178-8051,1432-2064},
     journal={Probab. Theory Related Fields},
      volume={97},
      number={3},
       pages={287\ndash 320},
         url={https://doi.org/10.1007/BF01195068},
      review={\MR{1245247}},
}

\bib{osgood}{article}{
      author={Osgood, W.~F.},
       title={Beweis der {E}xistenz einer {L}\"osung der
  {D}ifferentialgleichung {$\frac{{dy}}{{dx}} = f\left( {x,y} \right)$} ohne
  {H}inzunahme der {C}auchy-{L}ipschitz'schen {B}edingung},
        date={1898},
        ISSN={1812-8076},
     journal={Monatsh. Math. Phys.},
      volume={9},
      number={1},
       pages={331\ndash 345},
         url={https://doi.org/10.1007/BF01707876},
      review={\MR{1546565}},
}

\bib{salins-2022}{article}{
      author={Salins, Michael},
       title={Global solutions to the stochastic reaction-diffusion equation
  with superlinear accretive reaction term and superlinear multiplicative noise
  term on a bounded spatial domain},
        date={2022},
        ISSN={0002-9947,1088-6850},
     journal={Trans. Amer. Math. Soc.},
      volume={375},
      number={11},
       pages={8083\ndash 8099},
         url={https://doi.org/10.1090/tran/8763},
      review={\MR{4491446}},
}

\bib{s-2024-aop}{article}{
      author={Salins, Michael},
       title={Solutions to the stochastic heat equation with polynomially
  growing multiplicative noise do not explode in the critical regime},
        date={2024},
     journal={Ann. Probab.},
}

\bib{sz-2022}{article}{
      author={Shang, Shijie},
      author={Zhang, Tusheng},
       title={Stochastic heat equations with logarithmic nonlinearity},
        date={2022},
        ISSN={0022-0396,1090-2732},
     journal={J. Differential Equations},
      volume={313},
       pages={85\ndash 121},
         url={https://doi.org/10.1016/j.jde.2021.12.033},
      review={\MR{4362370}},
}

\end{biblist}
\end{bibdiv}

\end{document}